\documentclass[reqno]{amsart}
\usepackage{amssymb,amsmath,amsfonts,mathrsfs,fge}
\usepackage[mathcal]{euscript}
\usepackage{verbatim,epsfig,color}
\usepackage{enumitem}

 \theoremstyle{plain}
 \newtheorem{theorem}  {Theorem}  [section]
 \newtheorem{lemma}  [theorem]   {Lemma}
 \newtheorem{corollary}[theorem] {Corollary}
 \newtheorem{fact} [theorem] {Fact}
 \newtheorem{proposition} [theorem] {Proposition}

 \theoremstyle{definition}
 \newtheorem{remark}[theorem] {Remark}
 \newtheorem{definition}[theorem] {Definition}

\newcommand{\iso}[1][]{\buildrel {#1} \over \cong}

\newcommand{\ot}{{\leftarrow}}

 \usepackage{pgffor}
\foreach \x in {A,...,Z,a,...,z}{%
\expandafter\xdef\csname b\x\endcsname{\noexpand\ensuremath{\noexpand\mathbf{\x}}}
\expandafter\xdef\csname B\x\endcsname{\noexpand\ensuremath{\noexpand\mathbb{\x}}}
\expandafter\xdef\csname c\x\endcsname{\noexpand\ensuremath{\noexpand\mathcal{\x}}}
\expandafter\xdef\csname s\x\endcsname{\noexpand\ensuremath{\noexpand\mathscr{\x}}}
}

\let\polishlcross=\l
\def\l{\ifmmode\ell\else\polishlcross\fi}

\def\H{\ifmmode{H}\else\accent"07D\fi}

\newcommand{\full}{\{\zero,\unit\}}

\usepackage{dsfont}
\newcommand{\unitc}{M}
\newcommand{\unit}{\pmb{\infty}}
\newcommand{\zero}{\mathbf{0}}

\renewcommand{\setminus}{\mathbin{\mathpalette\rsetminusaux\relax}}
\newcommand\rsetminusaux[2]{\mspace{-4mu}
  \raisebox{\rsmraise{#1}\depth}{\rotatebox[origin=c]{-20}{$#1\smallsetminus$}}
 \mspace{-4mu}
}
\newcommand\rsmraise[1]{%
  \ifx#1\displaystyle .8\else
    \ifx#1\textstyle .8\else
      \ifx#1\scriptstyle .6\else
        .45%
      \fi
    \fi
  \fi}

\DeclareMathOperator{\id}{id}

\usepackage{bm}

\newcommand{\pid}[1]{\bm{\left\langle} #1 \right]}
\newcommand{\ppid}[1]{\bm{\left\langle} #1 \right)}

\newcommand{\Ind}{\BA}

\newcommand{\cDnt}{\cD_{\rm{nt}}}

\newcommand{\IM}{\bS}       
\newcommand{\IMc}{\bT}

\newcommand{\IMx}[2]{}

\newcommand{\ji}[2]{#1^{(\underline{#2})}}
\newcommand{\mi}[2]{#1^{(\overline{#2})}}

\newcommand{\ai}[2]{#1^{(#2)}}

\newcommand{\aia}{\ai\alpha{i}}
\newcommand{\ajb}{\ai\beta{j}}

\newcommand{\jis}{\ji{\alpha}{i}}

\newcommand{\jibp}{\ji{(\beta+1)}{j}}
\newcommand{\mis}{\mi{\beta}{j}}

\newcommand{\x}[1]{x^{(#1)}}

\newcommand{\gD}{D_\PC}
\newcommand{\gA}{A_\PC}
\newcommand{\rA}{\cA_{\rm{m}}}
\newcommand{\gK}{K_\PC}

\newcommand{\irrFam}{\cI}
\newcommand{\PC}{\cP}

\newcommand{\dC}{\cC^\infty}

\newcommand{\cupT}{\cT}

\newcommand{\CD}{\cC}

\newcommand{\II}[4]{[\ji{#3}{#1},\mi{#4}{#2}]}

\newcommand{\IIs}{\II{i}j\alpha\beta}

\newcommand{\JI}[1]{J(#1)}

\newcommand{\JId}[1]{J^\infty(#1)}

\begin{document}
\title{On the representations of finite distributive lattices}
\author{Mark Siggers}
\address{Kyungpook National University, Republic of Korea}
\email{mhsiggers@knu.ac.kr}
\thanks{Supported by the National Research Foundation (NRF) of Korea
  (2014-06060000)}

\begin{abstract}
  A simple but elegant result of Rival states that every sublattice $L$ of 
  a finite distributive lattice $\PC$ can be constructed from $\PC$ by removing 
  a particular family $\irrFam_L$ of its irreducible intervals. 

  Applying this in the case that $\PC$ is a product of a finite set $\CD$ of chains,
  we get a one-to-one correspondence $L \mapsto \gD(L)$ between the sublattices
  of $\PC$ and the preorders spanned by a canonical sublattice $\dC$ of $\PC$.   

  We then show that $L$ is a tight sublattice of the product of chains $\PC$  
  if and only if $\gD(L)$ is asymmetric. This yields a one-to-one 
  correspondence between the tight sublattices of $\PC$ and the posets spanned
  by its poset $\JI{\PC}$ of non-zero join-irreducible elements.  

  With this we recover and extend,  among other classical results, the
  correspondence derived from
  results of Birkhoff and Dilworth, between the tight embeddings of a finite
  distributive lattice $L$ into products of chains, and the chain decompositions
  of its poset $\JI{L}$ of non-zero join-irreducible elements.  

\end{abstract} 


 \maketitle

 \section{Introduction}     
   All lattices considered in this paper are finite and distributive. 
   For very basic notation, definitions, and concepts we refer
   the reader to \cite{Gr99}; many basic definitions are also given 
   at the start of Section \ref{sect:background}.   
 
   Classical results of Birkhoff and Dilworth, which we review in more 
   detail in Section \ref{sect:background}, show that any finite distributive
   lattice $L$ can be embedded, as a sublattice, into a product of chains.
   They further yield a one-to-one correspondence between the tight such  embeddings
   (those that preserve covers, and the the zero and unit elements),
   and chain decompositions of the poset $\JI{L}$ of
   non-zero irreducible elements of $L$.  

   In this paper, we consider the following question:   
   
   \begin{quote} What happens when we reverse our point of view, and ask about the 
   various sublattices of a given products of chains $\PC$ rather than the
   various embeddings of a particular lattice $L$ into products of chains?
  \end{quote}

   Starting with the product $\PC = \prod_{C \in \CD}C$ of a
   set $\CD$ of chains, $\JI{\PC}$ is the parallel sum of the
   chains we get from
   the chains of $\CD$ by removing their minimum elements.
   One main idea is that, for a tight sublattice $L$ of $\PC$, $\JI{L}$
   is isomorphic to an extension of $\JI{\PC}$. 
   See Figure \ref{fig:RivEx1} for an example. 
   If we consider only tight sublattices of $\PC$, this
   is the whole picture (Corollary \ref{cor:RisJ2}): $L \mapsto \JI{L}$ gives a 
   one-to-one correspondence of the tight sublattices of $\PC$ 
   and the spanned poset extensions of $\JI{\PC}$--
   extensions with the same set of elements as $\JI{\PC}$.
 \begin{figure}
   \centering{
    \setlength{\unitlength}{1mm}%
   \begin{picture}(100,40)%
    \put(0,5){\includegraphics[scale=.50]{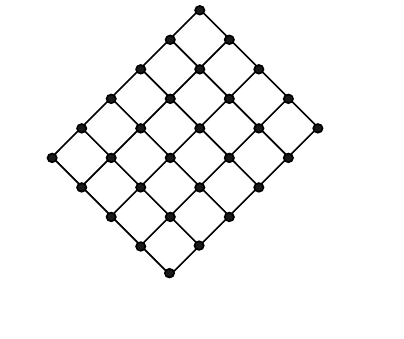}}%
    \put(13,6){$\PC$} 
    \put(30,8){\includegraphics[scale=.50]{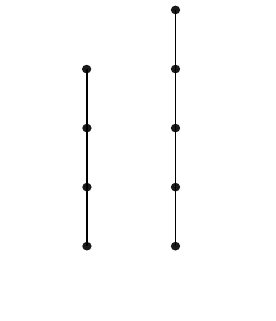}}%
    \put(37,6){$\JI{\PC}$}
    \put(57,5){\includegraphics[scale=.50]{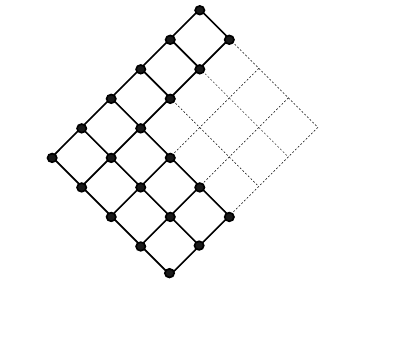}}%
    \put(70,5){$L$} 
    \put(82,8){\includegraphics[scale=.50]{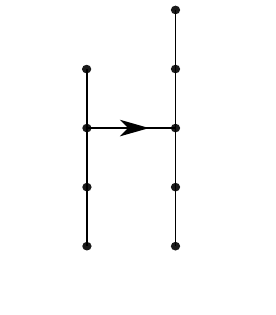}}%
    \put(90,5){$\JI{L}$}   
  \end{picture}
   \caption{ Product of chains $\PC$, sublattice $L$, and their posets 
             $\JI{\PC}$ and $\JI{L}$}
    \label{fig:RivEx1}  
   }
  \end{figure}
   But what if we consider non-tight embeddings / sublattices?

  As every lattice has a tight embedding into products of chains,
  non-tight embeddings are seldom considered.  
  However, in the paper \cite{Si15}, out of which this paper grew, 
  we found it useful to consider non-tight embeddings, or rather non-tight
  sublattices of products of chains. We needed a characterisation of {\bf all}
  sublattices of a product of chains.


  Using a result of Rival \cite{Ri74}
  which characterises the sublattices of $\PC$, we find  
  a correspondence $L \mapsto \gD(L)$ between the
  sublattices $L$ of $\PC$ and {\em preorders} (reflexive transitive
  relations) $\gD(L)$ extending $\JI{\PC}$.  
   To do this in full generality it is necessary to replace $\JI{\PC}$ with
   $\dC$, which we get from $\JI{\PC}$ by appending a zero and a unit.  
   Our main result yields Corollary \ref{cor:1to1}, a one-to-one
   correspondence between sublattices of $\PC$ and spanned preorder extensions of $\dC$. 
  
   We then characterise the tight sublattices $L$ 
   of $\PC$ as those for which $\gD(L)$ is a poset and so, after restricting
   back to the elements of $\JI{\PC}$, find that in this case $\gD(L)$ is $\JI{L}$.
   We further characterise 
   the so-called $\full$ and semi-direct sublattices of $\PC$ in terms of
   properties of $\gD$, and restricting our results to $\full$-sublattices,
   get, in Corollary \ref{cor:RisJ2} a one-to-one correspondence between the
   $\full$-sublattices of $\PC$ and the spanned preorder extensions of $\JI{\PC}$.

   Returning to the point of view of embeddings
   of a given lattice, we extend the classical correspondence between the tight
   embeddings of a lattice $L$ into products of chains and the 
   chain decompositions of $\JI{L}$. 
   We get, in Theorem \ref{thm:one2one5}, a one-to-one correspondence between
   embeddings of $L$ into
   products of chains, and surjective homomorphisms of `pointed unions of chains'
   to $\JId{L}$.

   In \cite{Ko83}, Koh used a clever construction to show that any distributive lattice $L$ can be represented as 
   the lattice $\rA(P)$ of maximal antichains of some poset $P$. 
   We finish off by showing that his construction arises naturally from our 
   point of view, and extend it to determine, in Corollary \ref{cor:antialldig}, 
   {\bf all} posets $P$ such that $L \iso \rA(P)$ for a given $L$.


  \section{Notation and Background}\label{sect:background}



  A lattice embedding $e: L \to L'$ is {\em a $\{\zero,\unit\}$-embedding}
  if it preserves the zero and unit elements: $e(\zero_L) = e(\unit_{L'})$
  and $e(\unit_L) = \unit_{L'}$.
  It is {\em tight} if it is a $\{\zero,\unit\}$-embedding and preserves covers:
  $x \prec y$ implies that $e(x) \prec e(y)$.

  An element $x$ of a lattice $L$ 
  is {\em join-irreducible}          
  if $x= a \vee b$ 
  implies that $x=a$ or $x = b$. Meet-irreducible elements are defined analogously.
  The set of non-zero join-irreducible elements of $L$ is denoted $\JI{L}$.
  It induces a subposet of $L$ which is also denoted by $\JI{L}$.

  For a subset $S$ of a lattice $L$, we let
  $\bigvee S = \bigvee_{x \in S} x$ be the join of the
  elements of $S$. We often write $\bigvee_L S$ to specify
  that the join takes place in $L$.   

  A subset $S$ of a poset is a {\em downset} or {\em ideal}
  if $x \in S$ and $y \leq x$ implies $y \in S$. 
  The minimum downset containing an element $x$ is denoted
  $\id x$.  

  A {\em chain} $C$ of length $n$ in a poset $P$ is subposet
  isomorphic to the linear order $Z_n$ on 
  the $n$ elements $\{0,1,\dots, n-1\}$.
  A {\em chain decomposition} of a poset $P$ is a partition of its elements
  into a family $\CD$ of chains $C_1, \dots, C_d$. 
  For a family $\CD = \{C_1, \dots, C_d\}$ of disjoint chains, the product   
   $\prod\CD := \prod_{i=1}^dC_i$ consists of all $d$-tuples $x = (x_1, \dots, x_d)$ 
  where $x_i \in C_i$ for each $i \in [d]$. It is ordered by $x \leq y$ if $x_i \leq y_i$ for 
  each $i$.

  \subsection{Embedding through Downsets}
 
  In \cite{Bi}, Birkhoff showed that $L \iso \cD(\JI{L})$ where $\cD(P)$, for a 
  poset $P$ is the family of downsets of $P$ ordered by
  inclusion.  Specifically, he showed the following. 
 
  \begin{theorem}\label{thm:birkEmb}\cite{Bi}
    The map $\IM:  x \mapsto \id x \cap \JI{L}$ 
   is an isomorphism of $L$ to $\cD(\JI{L})$. 
   Its inverse is  $S \mapsto \bigvee_L S$.
  \end{theorem}
  One can easily show that $P \iso \JI{\cD(P)}$ by observing that a downset in 
  $\cD(P)$ is join-irreducible if and only if it has a unique maximal element.  
  Thus Theorem \ref{thm:birkEmb} gives the following, which completes a one-to-one
  correspondence between finite posets and  finite distributive lattices.   
  \begin{corollary}\label{cor:revBirkEmb}
   If $\cD(P) \iso \cD(P')$ then $P \iso P'$. 
  \end{corollary} 
  
   For a chain decomposition $\CD$ of a poset let $\CD_0$ be the family
   of chains we get from the chains in $\CD$ by adding a new
    minimum element  to each. 
   In \cite{Di50}, Dilworth proved the following embedding theorem. 
   \begin{theorem}\label{thm:dilEmb}\cite{Di50}
     For any chain decomposition $\CD$ of a poset $P$
      the map $S \mapsto \bigvee_\PC S$
    is an embedding of $\cD(P)$ into  $\PC = \prod \CD_0$. 
  \end{theorem}

  With Theorem \ref{thm:birkEmb}, this immediately gives the following.

  \begin{corollary}\label{cor:birkEmb}
  For any decomposition $\CD$
  of $\JI{L}$ into chains, the map  
      \[ e_\CD: L \to \prod \CD_0: x \mapsto \textstyle{\bigvee_\PC} \IM(x) \]
  is an embedding of $L$ into $\PC = \prod \CD_0$.  
  \end{corollary}

  In \cite{La98}, Larson makes explicit a converse to Corollary 
  \ref{cor:birkEmb}, showing essentially the following. 
 
  \begin{theorem}\label{thm:La}\cite{La98}
    The embeddings $e_\CD$ of Corollary \ref{cor:birkEmb} are tight,
    and for every tight embedding $e$ there  is a chain decomposition
    $\CD$ of $\JI{L}$ such that $e = e_{\CD}$.
  \end{theorem}

  It is a trivial observation that different chain decompositions of $\JI{L}$
  yield different embeddings of $L$ into products of chains, so with Corollary
  \ref{cor:birkEmb}, this yields the following correspondence. 
  \begin{corollary}\label{cor:La}
    There is a one-to-one correspondence between the chain decompositions
     of $\JI{L}$ and the tight embeddings of $L$ into products
    of chains. 
  \end{corollary}

  \subsection{Embeddings through antichains}

 The correspondence taking a downset of a poset $P$ to its subset 
  of maximal elements gives one-to-one correspondence between the downsets
  $\cD(P)$ of $P$ and the antichains $\cA(P)$ of $P$.
  It is not hard to see that by defining the following ordering on
  $\cA(P)$, the correspondence can be extended to a lattice isomorphism:
  for antichains $I,I' \in \cA(P)$, set $I \leq I'$ if for all $a \in I$ there
  is some $a' \in I'$ such that $a \leq a'$.
  In \cite{Di60}, Dilworth showed the following, which is considerably harder.
  \begin{theorem}\label{thm:dilMaxAnt}\cite{Di60}
    The poset $\rA(P) \leq \cA(P)$ of all maximum sized antichains of a
    poset $P$ is a distributive lattice. 
  \end{theorem}

  In \cite{Ko83}, Koh showed a converse:  every finite distributive lattice $L$ is 
  $\rA(P)$ for some lattice $P$. In doing so he showed something stronger. 
  Koh defines a construction $(\JI{L},\CD) \mapsto P_\CD$ which yields a poset $P_\CD$ for every
  chain decomposition
  $\CD$ of the poset $\JI{L}$.
  He then showed the following. 

  \begin{theorem}\label{thm:koh}\cite{Ko83}
     For every finite distributive lattice $L$, and every chain
     decomposition $\CD$ of $\JI{L}$, $L \iso \rA(P_\CD)$.
  \end{theorem}

  \subsection{Rival's theorem}

   An {\em irreducible interval} of a lattice $L$ is the set
      $[\alpha,\beta] = \{ x \in L \mid \alpha \leq x \leq \beta \}$
  for any join-irreducible element $\alpha$ of $L$ and any 
  meet-irreducible element $\beta$.
  For a set $\irrFam$ of irreducible intervals of $L$ we let 
  $\bigcup \irrFam = \bigcup_{I \in \irrFam} I$.
  The set $\irrFam$ is  {\em closed} if $I \subseteq \bigcup\irrFam$
   implies $I \in \irrFam$.   
   The key to our results
   is the following theorem of Rival. 

  \begin{theorem}\label{thm:rival} \cite{Ri74}
    If $L$ is a sublattice of a lattice $L'$ of finite length, then 
            \[ L = L' \setminus  \bigcup \irrFam \]
    for some (closed) family $\irrFam$ of irreducible intervals of $L'$.
  \end{theorem}

  \subsection{Non-trivial Downsets}

  For a poset $P$ let $P_\zero$ be the poset we get from $P$ by adding a new
  minimum element $\zero$ and let $P^{\unit}$ be the poset we get from $P_\zero$
  by adding a new maximum element $\unit$. 
  A downset of a poset $P$ is {\em non-trivial} if it is a proper non-empty subset of $P$.  
  The following  has  been observed in several places,
  (see \cite{La98}.) 
 
  \begin{fact}\label{fact:nonEmptyD}
   Where $\cD_{\rm ne}(P)$ is the poset of {\em non-empty} downsets of a poset $P$,
   $\cD(P) \iso \cD_{\rm ne}(P_\zero)$. Where $\cDnt(P)$ is the poset of 
  {\em non-trivial} downsets of $P$,
    $\cD(P) \iso \cDnt(P^{\unit})$.
  \end{fact}

  As the construction $P \mapsto P^{\unit}$ is clearly invertible, this fact, with 
  Corollary~\ref{cor:revBirkEmb}, gives the following.
  \begin{corollary}\label{cor:revBirkEmbEx} 
   For posets $P$ and $P'$ with maximum and minimum elements, 
   if $\cDnt(P) \iso \cDnt(P')$ then $P \iso P'$. 
  \end{corollary}

  \section{Setup, and the definition of $\gD(L)$}

 For the rest of the paper $\PC$ is always a product  $\PC = \prod\CD$ of
 finite chains $\CD = \{C_1, \dots, C_d\}$. When $L$ is a sublattice of $\PC$,
 Theorem \ref{thm:rival} yields a unique closed family $\irrFam_L$ such that
 $L = \CD \setminus \bigcup \irrFam_L$.
  The join irreducible elements of $\PC$ are clearly of the form 
    \[ \ji{\alpha}{i} := (\underbrace{0,0,\dots,0,\alpha}_{i},0,0,\dots,0) \]
  for $i \in [d]$ and $\alpha \in C_i$ and the meet-irreducible elements
  are 
   \[ \mi\beta{j}  := (\unitc_1,\unitc_2,\dots,\unitc_{j-1},\beta,\unitc_{j+1}, \dots,\unitc_d) \]
  for $j \in [d]$ and $\beta \in C_j$, where $\unitc_i$ is the maximum elemement of $C_i$.
  The irreducible intervals of $\PC$ are exactly the intervals
  $\IIs$
  for $i,j \in [d]$, $\alpha \in C_i$, and $\beta \in C_j$. 
  Note that $\ji0i = \zero$ for all $i$, and $\mi\infty{j} = \unit$ for all $j$.
  Given $\CD$ and $\PC = \prod\CD$, let $\dC = \JId{\PC}$ be the subposet of $\PC$ induced
  by $\JI{\PC} \cup \{\zero, \unit\}$. 
  As $\unit$ is the element of $\PC$ above $\ji{\unitc_i}{i}$ and $\zero$ is the 
  element below $\mi{0}{j}$, it seems reasonable to let 
  $\ji{(\unitc_i+1)}{i} = \unit$ and $\mi{(-1)}{j} = \zero$. We do this occasionally to 
  make definitions and proofs more compact.

  
 Recall that a poset, and in particular $\dC$, is a {\em reflexive digraph}:
  a reflexive relation.
  We view posets as reflexive digraphs,  saying $(x,y)$ is an {\em arc}, or writing
  $x \to y$,  to mean $x \leq y$.
  A {\em spanned extension} of $\dC$ is any digraph on the elements of $\dC$,
  which contains it. A {\em preorder} is a transitive reflexive digraph.

    \begin{definition}[$\gD(\irrFam)$ and $\gD(L)$] 
     \label{const:D}
       For any family $\CD = \{C_1, \dots, C_d\}$ of chains and any
       family of irreducible intervals $\irrFam$ of $\PC = \prod \CD$,
       let 
        $\gD(\irrFam)$ be
       the spanned extension of $\dC$ that we get from $\dC$ by letting
       $\jibp \to \jis$ for each $\IIs \in \irrFam$.
       For a sublattice $L$ of $\PC$, let   $\gD(L) = \gD(\irrFam_L)$.
    \end{definition} 

  In Proposition \ref{prop:Rtrans} we will show that because 
  $\irrFam_L$ is closed, $\gD(L)$ is a preorder.
  The Hasse diagram is not uniquely defined for a preorder, as it is for a poset,
  but can still be used to depict one-- the preorder is still the transitive
  closure of the Hasse diagram.  We therefore depict $\gD(L)$ with a Hasse diagram.  
      
   The top half of Figure \ref{fig:RivEx2} shows the same lattices
   $\PC = Z_5 \times Z_6$ and
   $L_1 = \PC \setminus \II2132$ as are shown in Figure \ref{fig:RivEx1}, 
   but now, instead of $\JI{\PC}$ and $\JI{L_1}$, it shows with them the digraph
   $\dC$ and its spanned extension $\gD(L_1)$. 

 \begin{figure}
   \centering{
    \setlength{\unitlength}{.9mm}%
   \begin{picture}(110,50)%
    \put(0,5){\includegraphics[scale=.63]{./LatticeP.pdf}}%
    \put(10,3){$\PC = Z_5 \times Z_6$} 
    \put(11.5,9){$(0,0) = \zero$}
    \put(-14,25){$(4,0) = \ji41$}
    \put(39,30){$(0,5) = \ji52$}
    \put(20,46){$(4,5) = \unit$}
    \put(75,8){\includegraphics[scale=.63]{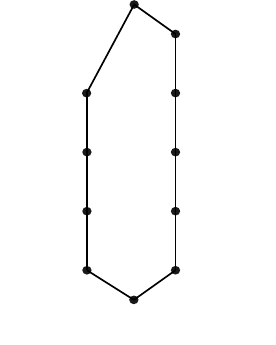}}%
    \put(87,3){$\dC$} 
    \put(79,15){$\ji11$}
    \put(97,15){$\ji12$}
    \put(79,36){$\ji41$}
    \put(97,45){$\ji52$}
    \put(89,50){$\unit$}
    \put(90,10){$\zero$}
  \end{picture}

  \begin{picture}(110,55)%
   \put(2,5){\includegraphics[scale=.63]{./LatticeL1.pdf}}%
    \put(8,3){$L_1 = \PC \setminus \II2132$} 
    \put(20,9){$\zero$}
    \put(1,25){$\ji41$}  
    \put(34,22){$\ot \ji32$}
    \put(34,36){$\ot \mi21$}
    \put(23,46){$\unit$}
    \put(75,8){\includegraphics[scale=.63]{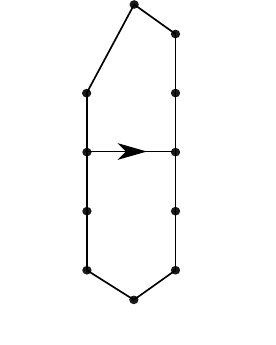}}%
    \put(83,3){$\gD(L_1)$} 
    \put(79,15){$\ji11$}
    \put(97,15){$\ji12$}
    \put(79,31){$\ji31$}
    \put(97,31){$\ji32$}
    \put(89,50){$\unit$}
    \put(90,10){$\zero$}
    
  \end{picture}%

  \begin{picture}(110,55)%
   \put(2,5){\includegraphics[scale=.63]{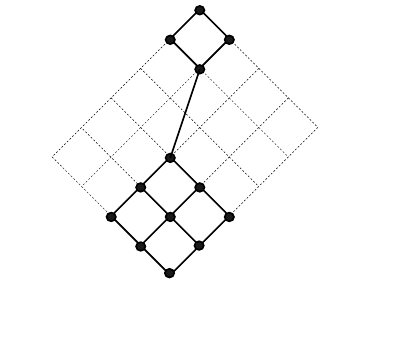}}%
    \put(8,3){$L_2 = L_1 \setminus \II1233$} 
    \put(20,9){$\zero$}
    
    \put(7,36){$\mi32 \to$}
    \put(0,22){$\ji31 \to$}
    \put(34,22){$\ot \ji32$}
    \put(34,36){$\ot \mi21$}
    \put(23,46){$\unit$}
    \put(75,8){\includegraphics[scale=.63]{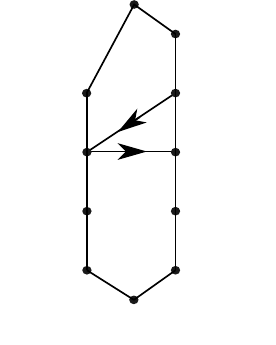}}%
    \put(83,3){$\gD(L_2)$} 
    \put(79,31){$\ji31$}
    \put(97,15){$\ji12$}
    \put(97,38){$\ji42$}
    \put(89,50){$\unit$}
    \put(90,10){$\zero$}
    
  \end{picture}%

  \begin{picture}(110,55)%
   \put(2,5){\includegraphics[scale=.63]{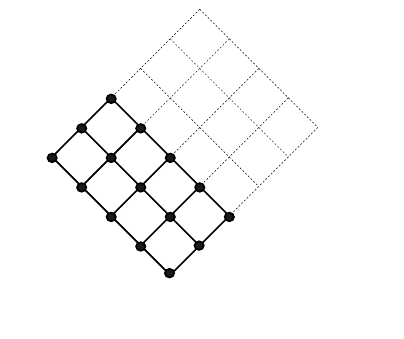}}%
    \put(8,3){$L_3 = \PC \setminus \II2134$} 
    \put(20,9){$\zero$}
    \put(1,25){$\ji41$}  
    \put(34,22){$\ot \ji32$}
    \put(27,43){$\ot \mi41 = \unit$}
    \put(75,8){\includegraphics[scale=.63]{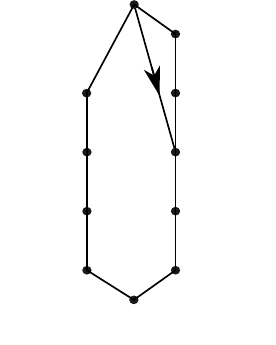}}%
    \put(83,3){$\gD(L_3)$} 
    \put(80,15){$\ji11$}
    \put(97,15){$\ji12$}
    \put(80,36){$\ji41$}
    \put(97,45){$\ji52$}
    \put(89,50){$\unit$}
    \put(90,10){$\zero$}
    
  \end{picture}%
  }
   \caption{ Sublattices of $\PC = Z_5 \times Z_6$ and their digraphs $\gD$}
    \label{fig:RivEx2}  
  \end{figure}



  An {\em $(x,y)$-path} $v_1 \to v_2 \to \dots \to v_\l$  in a digraph $D$
   is a sequence of elements 
   $x = v_1, v_2, \dots, v_\l = y$ such that $v_i \to v_{i+1}$ for each $i$. 
  The following is a straight-forward extension of the poset definition of 
  downset to digraphs.

  \begin{definition}\label{def:termSet}
   A subset $S \subseteq D$ of vertices of a digraph $D$ is a {\em downset}
     (or an {\em initial set}) if for all $y$ in $S$ and all $(x,y)$-paths in $D$, 
     $x$ is also in $S$.
     A downset set $S$ of $D$ is a {\em non-trivial downset} if it is 
     a non-empty proper subset of $D$. 
     For a vertex $x$ of $D$ let $\id x$ be the smallest downset of $D$ containing $x$. 
   \end{definition}

   \begin{definition}[$\cD(D)$ and $\cDnt(D)$]\label{def:Term}
     For any digraph $D$, let $\cD(D)$ be the family of downsets of $D$ ordered by inclusion,
     and $\cDnt(D)$ be the subfamily of non-trivial downsets.   
   \end{definition}


  The following simple case of our main theorem is
   clear from the top example in Figure \ref{fig:RivEx2}, and indeed, as $\PC$
   is a lattice and $\dC = \JId{L}$, it is essentially just a simple case of
   Theorem \ref{thm:birkEmb}, 
   using Fact \ref{fact:nonEmptyD}.

   \begin{lemma}\label{lem:baseTermSetIso}
      Where $\PC$ is a product of chains, the map    
         \[ \IMc: \PC \to \cDnt(\dC) : x \mapsto \id x \cap \dC \]
      is an isomorphism. 
     The inverse is $S \mapsto \bigvee_\PC S$.  
   \end{lemma}  

  Before we prove our main theorem,  that $\IMc$ is an 
  isomorphism from $L$ to $\cDnt(\gD(L))$ for {\bf any} sublattice of $\PC$,  
  we motivate the definition of $\gD(L)$ with a couple of examples.  
   It arises from observing what we must do to $\dC$ so that 
   $\IMc$ remains an isomorphism when we remove some irreducible
   interval $\IIs$ from $\PC$. 

   Referring to the second example in Figure \ref{fig:RivEx2}, when we remove $\II2132$ from 
   $\PC$, we see that we should add the $\ji32 \geq \ji31$ to $\dC$
   to maintain the isomorphism. Indeed, this ensures that, for example, the set 
  $\IMc((2,3)) = \{\ji21, \ji11, \zero, \ji12, \ji22, \ji32\}$ is no longer
  a downset.  In general it `kills' any downset $S$ containing $\ji32$
  but not $\ji31$. These are exactly the downsets $\IMc(x)$ for $x \in \II{2}{1}{3}{2}$.

  The third example in the figure shows a non-tight sublattice $L_2$ of $\PC$,
 observe that the corresponding  extension $\gD(L_2)$ of $\dC$
 is not a poset. The fourth example exhibits the necessity of $\zero$ and  $\unit$
 in $\gD$.

  \subsection{The closure of $\irrFam$ and transitivity of $\gD(\irrFam)$} 
  Observe that irreducible intervals of $\PC = \prod \CD$ of the
  form $[\ji\alpha{i}, \mi\beta{i}]$, where $\beta < \alpha$, are empty.
  By definition they are contained in any closed family $\irrFam$ of irreducible
  intervals. This corresponds to $\gD(\irrFam)$ containing $\dC$.

  \begin{lemma}\label{lem:tech}
    Let $D = \gD(\irrFam)$ for some sublattice $L = \PC \setminus \bigcup \irrFam$ of $\PC = \prod \CD$. 
    The following are equivalent for  
    $\alpha \in C_i \in \CD$ and  $\beta \in C_j \in \CD$. 
     \begin{enumerate}
       \item  $\IIs \subseteq \bigcup \irrFam$.
       \item  For all $x \in L$, $x_i \geq \alpha \Rightarrow x_j > \beta$.
       \item  There is a $(\jibp,\jis)$-path in $D$.   
     \end{enumerate}
   
  \end{lemma}
  
      \begin{proof}
       The equivalence of (i) and (ii) is immediate as both are 
       equivalent to the statement that
       there are no elements $x \in L$ with  $\alpha \leq x_i$ and
       $x_j \leq \beta$. So we show the equivalence of (ii) and (iii). 
 
       On the one hand, assume that 
   \[\ji{(\beta+1)}{s_\l} = \ji{\alpha_{s_\l}}{s_\l} \to \dots \to \ji{\alpha_{s_2}}{s_2}\to\ji{\alpha_{s_1}}{s_1} \]
       is an $(\ji{(\beta+1)}{s_\l},\ji{\alpha_{s_1}}{s_1})$-path in $D$. 
       By definition of $D = \gD$ this means  that 
       $\II{{s_i}}{{s_{i+1}}}{{\alpha_{s_i}}}{(\alpha_{s_{i+1}} - 1)} \in \irrFam$ for each $i$.
        Let $x \in L$ be such that $x_{s_1} \geq \alpha_{s_1}$.
        By the equivalence of (i) and (ii) this implies 
        $x_{s_2} > \alpha_{s_2} - 1$, which in turn implies  
        implies $x_{s_3} > \alpha_{s_3} - 1$, etc., until we get that 
        $x_{s_\l} > \alpha_{s_\l} - 1 = \beta$, as needed. 

       On the other hand, assume that there is no such path 
       from $\ji{\alpha_{s_\l}}{s_\l}$ to  $\ji{\alpha_{s_1}}{s_1}$. 
       We will find $x$ in $L$ with 
       $x_{s_1} \geq \alpha_{s_1}$ and $x_{s_\l} < \alpha_{s_\l} = \beta +1$.
       Indeed, for $k \in [d]$ let $x_k$ be the maximum value for which 
       there is a  path in $D$ from $\ji{x_{k}}{k}$ to $\ji{\alpha_{s_1}}{s_1}$.
       If no such path exists, let $x_k$ be $0$. 
       Clearly $x_{s_1} \geq \alpha_{s_1}$, 
       and by the assumption
       that there is no path from $\ji{\alpha_{s_\l}}{s_\l}$ to $\ji{\alpha_{s_1}}{s_1}$
       we have 
       that $x_{s_\l} < \alpha_{s_\l}$.
       We have just to show that $x$ is in $L$. 

       If $x$ is not in $L$, then it is in some 
       $\II{i}{j}{\alpha_i}{\beta_j} \in \irrFam$. 
       As $\II{i}{j}{\alpha_i}{\beta_j} \in \irrFam$, we have that  
       $\ji{(\beta_j + 1)}{j} \to \ji{\alpha_{i}}{i}$ is in $D$.
       As $x \in \II{i}{j}{\alpha_i}{\beta_j}$, we have that $\alpha_i \leq x_i$ 
        and $x_j \leq \beta_j$. 
       Now by the choice of $x$,  $\alpha_i \leq x_i$ implies that there is an path from 
       $\ji{\alpha_i}{i}$ to $\ji{\alpha_{s_i}}{s_i}$, so the arc
       $(\ji{(\beta_j + 1)}{j}, \ji{\alpha_{s_i}}{s_i})$
       gives us a path from $\ji{(\beta_j + 1)_{s_j}}{s_j}$ to $\ji{\alpha_{s_i}}{s_i}$.
       But then $\beta_j + 1 \leq x_j$  which contradicts
       $x_j \leq \beta_j$.   
     \qed\end{proof}

    As $\irrFam_L$ is closed, the following proposition allows us to assume that
    the digraph $\gD(L)$, for any sublattice $L$ of $\PC$, is a preorder.

    \begin{proposition}\label{prop:Rtrans}
      Let $D = \gD(\irrFam)$ for some sublattice $\PC \setminus \bigcup \irrFam$ of $\PC$. 
      Then $\irrFam$ is closed if and only if $D$ is transitive. 
     \end{proposition}
     \begin{proof}
      On the one hand, assume that $\irrFam$ is closed. 
      Then we can replace (i) of Lemma \ref{lem:tech} with
       $\IIs \in \irrFam$. 
      Now that $\ji\gamma{k} \to \ji\beta{j}$ and
      $\ji\beta{j} \to  \ji\alpha{i}$ in $D$ gives, by 
      Lemma \ref{lem:tech}, that $\II{i}{k}\alpha\gamma \in \irrFam$
      which means $\ji\gamma{k} \to \ji\alpha{i}$ in $D$. 
      So $D$ is transitive. 

    On the other hand, assume that $\irrFam$ is not closed. 
    Then there is some $\IIs$ in $\bigcup \irrFam$ but not 
    in $\irrFam$. So there is a $(\jis,\mis)$-path
    in $D$, while  $(\jis,\mis)$ is not 
    in $D$.     
    \qed\end{proof}




  \section{ Main Theorem: Sublattices of $\PC$ as downset lattices} \label{sect:downsets}


   As with posets, it is clear for digraphs that unions and intersections of 
   downsets are downsets, so $\cDnt(D)$ is indeed a poset for any digraph $D$. 
   We only need this when $D$ is a spanned extension of some $\dC$, 
   but in this case, we get more.

 \begin{theorem}\label{thm:mainIsoDownsets}
        Let $L$ be a sublattice of a  product of chains $\PC$, and $D = \gD(L)$.
        The map $\IMc: L \to \cDnt(D): x \mapsto \id x \cap \dC$ is a lattice isomorphism. 
  \end{theorem}
  \begin{proof}
     As it simplifies induction, we prove the slightly more general result that $\IMc$
     is an isomorphism for $D = \gD(\irrFam)$ where $\irrFam$ is any family of irreducible
     intervals of $\PC$ such that $L = \PC \setminus \bigcup\irrFam$.        
     In the case that $\irrFam$ is empty, or contains only empty families,
     we have $L = \PC$ and $D = \dC$, 
     so the isomorphism is given in Lemma \ref{lem:baseTermSetIso}.%
     
     In the general case, we first 
     observe that $\cDnt(D)$ is a subposet of $\PC =  \cDnt(\dC)$.
     To see that it is a subset, observe that adding arcs to $D$ can only 
     remove downsets, not create new ones: if a downset in $D$
     contains $\ji{x}{i}$,
     then it must contain $\jis$ for all $\alpha < x$. So
     any downset set of $D$ is the union, over its elements $x$,  of the downsets
     $\id x$ they generate in $\dC$.   
     As  $\cDnt(D)$ is also ordered by inclusion, its order is induced from 
     $\cDnt(\dC)$, so is a subposet.

     Now it is enough to show that  $\IMc: L \to \cDnt(D)$ a bijection.  We do 
     this by induction on the size of $\irrFam$.
      Let $\IIs \in \irrFam$, $\irrFam' = \irrFam \setminus \IIs$, $L' = \PC \setminus \bigcup \irrFam'$, and 
      $D' = \gD(\irrFam')$. 
      By induction we have that $\IMc: L' \to \cDnt(D')$ is a bijection.

      We must show for $\IMc(x) \in \cDnt(D')$, that $\IMc(x) \not\in \cDnt(D)$ if and
      only if $x \in \IIs$.  Now $\IMc(x)$ being a downset in $D'$, but not in 
      $D$ which we get from $D'$ by adding the arc $( \jibp, \jis )$,
      means exactly that
      $\jis \in \IMc(x)$ and $\jibp \not\in \IMc(x)$. This is to say 
      $\alpha \leq x_i$ but $(\beta + 1) > x_j$, which means that $x \in \IIs$, 
      as needed.  
   \qed\end{proof}

    By the theorem, we have that any sublattice of $\PC$ can be expressed as the 
    lattice  $\cDnt(D)$ of non-trivial downsets of some spanned extension $D$ of $\dC$. 
    On the other hand, it is clear that for a spanned extension $D$ of $\dC$ the
    family 
        \[ \irrFam_D = \{ \IIs \mid (\jibp,\jis) \in D \setminus \dC\}\] 
    is such that $\cDnt(D)$ is the sublattice $\PC \setminus \bigcup \irrFam_D$
    of $\PC$. While several digraphs $D$ may yield the same sublattice
    of $\PC$, as may families of intervals have the same union, it is clear that
    there is a unique closed family of intervals defining a given sublattice. 
    So by  Proposition \ref{prop:Rtrans} we get the following. 

    \begin{corollary}\label{cor:1to1}
     The map $L \mapsto \gD(L)$ 
     gives a one-to-one
     correspondence between the sublattices of $\PC$ and the spanned preorder
     extensions of $\dC$. 
    \end{corollary}

    \section{ Classification of sublattices }\label{sec:char}

    Recall that a sublattice $L$ of $\PC$ is a 
   $\{\zero,\unit\}$-sublattice if it contains the extremal elements
    $\zero_\PC$ and $\unit_\PC$ of $\PC$.
    It is {\em subdirect} if for each $i \in [d]$, the projection 
    $\pi_i: L \to C_i$ is surjective; this is necessarily a
    $\full$-sublattice.
    A $\full$-sublattice is {\em tight} if its covers are covers of $\PC$.

  It was shown in \cite{La98} that every tight sublattice of a product
  of chains is subdirect. The converse was also claimed, but the proof was
  flawed: indeed, the lattice 
     $L = Z_3 \times Z_3 \setminus (\II{1}{2}{2}{1} \cup \II{2}{1}{2}{1})$
  shown in Figure~\ref{fig:subNotTight} is a subdirect sublattice of 
  $Z_3 \times Z_3$,
  but not tight.
  One notices in this example that $\gD(L)$ has a cycle. 
  We will see that this is indicative of non-tight sublattices.

  \begin{figure}
  \begin{center}
  
  \setlength{\unitlength}{300bp}%
  \begin{picture}(.8,0.46874999)%
    \put(.045,.11){\includegraphics[scale=.82]{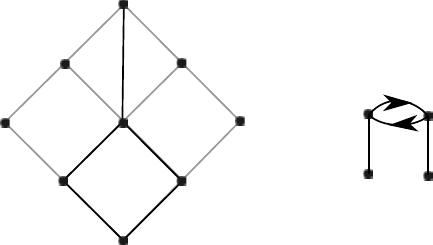}}%
    \put(0.50,0.27){\color[rgb]{0,0,0}\makebox(0,0)[lb]{\smash{2}}}%
    \put(0.62,0.27){\color[rgb]{0,0,0}\makebox(0,0)[lb]{\smash{2}}}%
    \put(0.62,0.19){\color[rgb]{0,0,0}\makebox(0,0)[lb]{\smash{1}}}%
    \put(0.50,0.19){\color[rgb]{0,0,0}\makebox(0,0)[lb]{\smash{1}}}%
    \put(0.28,0.12){\color[rgb]{0,0,0}\makebox(0,0)[lb]{\smash{$L$}}}%
    \put(0.52,0.12){\color[rgb]{0,0,0}\makebox(0,0)[lb]{\smash{$\gD(L)$}}}%
  \end{picture}%
 \caption{Non-tight subdirect sublattice of $Z_3 \times Z_3$}\label{fig:subNotTight}
  \end{center}
   
\end{figure}

    The {\em height} of a lattice is the length of a maximum ascending 
    chain of covers from $\zero$ to $\unit$. It is well known that for
    a distributive lattice,
    every cover is in a maximum ascending chain. It follows that 
    a sublattice $L$ of $\PC$ is tight if and only if $L$ and $\PC$ have the
    same height. 
    An arc $(x,y)$ in a digraph $D$ is an {\em in-arc} of $y$ and an {\em out-arc}
    of $X$.


    \begin{lemma}\label{lem:full}
     Let $L$ be a sublattice of $\PC$ and $D = \gD(L)$.
      \begin{enumerate}
         \item 
     $L$ is a $\{\zero,\unit\}$-sublattice  if and only if 
     $\zero$ is a source (has no in-arcs) in $D$ and 
     $\unit$ is a sink (has no out-arcs). 
         \item $L$ is subdirect if and only if $D$ has no {\em down edges}:
               those of the form $\ji\alpha{i} \to \ji{(\alpha-1)}i$.
         \item $L$ is tight if and only if $D$ is a poset.
       
     \end{enumerate}
             
    \end{lemma} 
    \begin{proof}
         We have by Theorem \ref{thm:mainIsoDownsets} that $\IMc : L \to  \cDnt(D)$
          is an isomorphism. 

     \noindent    (i)  Clearly $\IMc(\zero) = \{ \zero \}$ is a downset of $D$ if and only if
         $\zero$ is a source in $D$, and 
         $\IMc(\unit) = \dC \setminus \{\unit\}$ is a downset if and only if
         $\unit$ is a sink. The result follows.

     \noindent (ii). We have $\ji{\alpha}{i} \to \ji{(\alpha-1)}{i}$ in $D$ 
       if and only if there is no downset $\IMc(x)$ of $D$ containing 
      $\ji{(\alpha-1)}{i}$ but not $\ji{\alpha}{i}$. 
       This is if and only if there is no element $x \in L$ with 
       $x_i = \alpha-1$. 
         
   \noindent  (iii).  First assume that $L$ is tight, and assume,
      towards contradiction that $D$ contains a cycle
       $\ji{\alpha_{1}}{i_1} \to \ji{\alpha_2}{i_2} \dots \to 
      \ji{\alpha_\l}{i_\l} \to \ji{\alpha_1}{i_1}$, for some $\l \geq 2$. 
     We show that no vertex $x$ in $L$ has $x_{i_1} = \alpha_1$, which 
     contradicts the fact that $L$ is tight.  
     Indeed, if $x$ did have $x_{i_1} = \alpha_1$, then
     $x_{i_1} \geq \alpha_1$ so by Lemma \ref{lem:tech} $x_{i_2} > \alpha_2$ so
     $x_{i_3} > \alpha_3$ etc., and we get that $x_{i_1} > \alpha_1$,
      a contradiction. 
      
     On the other hand, assume that $D$ has no cycles. Then its vertices can be
     ordered so that   all arcs go up.  Visiting vertices from the bottom of this ordering,
      one at a time, we get an ascending  walk in $L \iso \cDnt(D)$ from $\zero$ to $\unit$
     of  size $|D|$, showing that $L$ has height equal to the height of $\PC$.
     Thus $L$ is a tight sublattice.          
    \qed\end{proof}

     For any preorder $D$ with named elements $\zero$ and $\unit$, let $D^*$ be the
     subgraph induced by $D \setminus \{ \zero, \unit \}$.      
    A spanned extension $\gD(L)$ of $\dC$ yields a spanned extension 
    $\gD^*(L)$ of $\CD^* := (\dC)^*$, so with Lemma \ref{lem:full}(i),
    Corollary \ref{cor:1to1}  restricts to the following. We use Lemma \ref{lem:full}(iii)
    in the second statement.

    \begin{corollary} \label{cor:RisJ2}
      The map $L \mapsto \gD^*(L)$ gives a one-to-one correspondence between the
      $\full$-sublattices of $\PC$ and the spanned preorder extensions of $\CD^*$.
      It gives a  one-to-one correspondence
      between the tight sublattices of $\PC$ and the spanned poset extensions of $\CD^*$.
    \end{corollary}

     Lemma \ref{lem:full}(i) also gives us the following.

     \begin{lemma}\label{lem:fulltech}
     Let $L$ be a $\full$-sublattice of $\PC$ and $D = \gD(L)$.
     Then the map $\cDnt(D) \to \cD(D^*): S \mapsto S \setminus \{\zero\}$
     is an isomorphism. 
     \end{lemma}
     \begin{proof}
         We show that $S \mapsto S \setminus \{\zero\}$ is a bijection; as the orders
         in $\cDnt(D)$ and $\cD(D^*)$ are subset orders, it then clearly preserves
         order.  

         As $\{\zero\}$ is in any non-trivial downset $D$ and has no in-arcs, 
         it is clear that for any non-trivial downset $S$ of $D$,
        $S \setminus \{\zero\}$
         is downset of $D^*$. 
         We must show that there are no other downsets of $D^*$.
       Assume $S$ is a downset in $D^*$ but $S \cup \{\zero\}$ is not a downset of 
       $D$.  Then $S$ must have some in-arc in $D$ that does not exist in $D^*$.
       This can  only be an arc from $\unit$ but by Lemma \ref{lem:full}(i), 
       there are no such arcs.
     \qed\end{proof}

     Using Lemma \ref{lem:fulltech} to extend the isomorphism 
     $\IMc: L \to \cDnt(D)$
     of Theorem \ref{thm:mainIsoDownsets} to an isomorphism
  $\IM: L { \to}\cDnt(D) \to \cD(D^*)$, we get the following.
     
 \begin{corollary}\label{cor:mainIsoDownsetsFull}
        Let $L$  be a $\full$-sublattice of $\PC = \prod\CD$ and let $D^* = \gD^*(L)$.
        The map $\IM: L \to \cD(D^*): x \mapsto \id x \cap \CD$ 
        is a lattice isomorphism. 
  \end{corollary}
     
  That Corollary \ref{cor:mainIsoDownsetsFull} and so Theorem \ref{thm:mainIsoDownsets}
  are extensions of Theorem \ref{thm:birkEmb} become more explicit with the following
  observation.

    \begin{lemma}\label{lem:RisJ}
       If $L$ is a tight sublattice of $\PC$,
       then 
        $\gD^*(L) \iso \JI{L}$.
     \end{lemma}
     \begin{proof}
 If $L$ is a tight sublattice of $\PC$ then $\gD(L)$ is
   a poset by Lemma \ref{lem:full}(iii), and so 
   $\gD^*(L)$ is too. 
   Thus $L \iso \cDnt(\gD(L)) \iso \cD(\gD^*(L))$ using 
   Lemma \ref{lem:fulltech}, and the statement follows by 
   Corollary \ref{cor:revBirkEmb}.
     \qed\end{proof}

   \section{Back to Embeddings}\label{sect:Koh}

    We now view a sublattice $L$ as the image of an embedding 
    $e:L \to \PC$ and the subgraph $\dC$ of $D$ as the image of a bijective
    homomorphism $\phi: \dC \to D$. As such, Corollary \ref{cor:1to1} gives a
    correspondence  between embeddings of lattices into
    $\PC$ and surjective homomorphisms of $\dC$ to preorders $D$.

    
    Formally, a {\it pointed union of chains} is any poset that we can form from
    a parallel union of disjoint chains by identifying the minimum elements
    of the chains and adding  a new maximum element $\unit$-- that is, any
    poset that can be represented as $\dC$ for some product $\PC = \prod \cC$ of
    chains. A {\it pointed chain decomposition} of a preorder $D$ is a bijective
    homomorphism $\phi:\dC \to D$ from a pointed union of chains. 

    Given a pointed chain decomposition $\phi: \dC \to D$ of $D$. 
    the map $\IMc$ of Theorem~\ref{thm:mainIsoDownsets} becomes
    $\IMc: L \to \cDnt(D): x \mapsto \phi(\id x \cap \dC)$ and its inverse
    becomes $\IMc^{-1}: S \mapsto \bigvee_\PC \phi^{-1}(S)$. This
    gives the following. 

   \begin{corollary}\label{cor:MainTermRev}
    For any pointed chain decomposition $\phi: \dC \to D$ of a preorder $D$
    the map $e_{\phi}: S \mapsto \bigvee_\PC \phi^{-1}(S)$
    is an embedding of $\cDnt(D)$ into $\PC$.
   \end{corollary}

   By Lemma \ref{lem:full}, $e_{\phi}$ is $\full$-embedding if and only 
   if $\zero$ is a source of $D$ and $\unit$ is a sink. In this case
   $\phi^{-1}(\zero) = \{\zero\}$ and $\phi^{-1}(\unit) = \{\unit\}$.
    Similar to the observation we made before Corollary \ref{cor:RisJ2},
    this means that $\phi(\CD^*)$ is a chain decomposition of $D^*$.
   As $\CD^* = (\dC)^*$ is the union of the chains of $\CD$, with
    their zero elements removed, and $\CD^*_0$ is the family of
    chains we get from the chains in $\CD^*$ by adding a new minimum
    element to each, $\CD^*_0 = \CD$.

 The following is thus a corollary of Corollary \ref{cor:mainIsoDownsetsFull}.
 Identifying $\phi(\CD^*)$ with $\CD^*$, it is a  generalization of
  Theorem \ref{thm:dilEmb}.

     \begin{corollary}\label{cor:MainTermRevFull}
       For any chain decomposition $\phi:\CD^* \to D^*$ of a preorder $D^*$ let 
      $\CD = \CD^*_0$, the map
       $e_\phi: S \mapsto  \bigvee_\PC S$ 
      is a $\full$-embedding of $\cD(D^*)$ into $\PC = \prod \CD$. 
      \end{corollary} 

    Now, for given distributive lattice $L$, there are several preorders
    $D$ for which $L \iso \cDnt(D)$, so which yield embeddings of $L$ into
    a product of chains.  To get back to a canonical $D$, we first must 
    take a quotient.

    \subsection{A Canonical Quotient of $\gD(L)$}

    The {\em non-preference} relation '$\sim$' defined on a   
  preorder $D$ by letting $x \sim y$ if $x$ and $y$ are
    in a directed cycle is an equivalence relation.
    The quotient $P_D:=D/\sim$ is  well known to be a poset.
    Classes of $P_D$  satisfy  $[a] \leq [b]$ if and only
    if $a \leq b$ in $D$. 
    
    Now the {\em quotient homomorphism} 
    $q : D \to P_D : a \mapsto [a]$
    induces maps between $2^D$ and $2^{P_D}$. For 
    $T \subseteq D$, we let $[T]:= \{[x] \mid x \in T\} \subseteq P_D$, and
    for $S \subseteq P_D$ we let $\bigcup S = \bigcup_{[x] \in S}[x]$.
    Clearly $[\bigcup S] = S$ and $\bigcup [T] = T$.

    \begin{lemma}\label{lem:quotient}
      Let $D$ be a preorder, then  
       $\cDnt(D) \to \cDnt(P_D): T \mapsto [T]$
      is an isomorphism with inverse $S \mapsto \bigcup S$.   
    \end{lemma}
    \begin{proof}
      Clearly $[T] \in P_D$ is a downset if $T \subset D$ is, and $\bigcup S \subset D$
      is if $P_D$ is, so $T \mapsto [T]$ is a bijection. 
      That is preserves order is also clear as  the order
       on both posets is inclusion, and $[T]$ simply
       partitions a set $T$  with respect to an underlying
       fixed partition of $D$. 
    \qed\end{proof}

   Recall that for a lattice $L$, $\JId{L}$ is the subgraph induced by 
   $\JI{L} \cup \{\zero, \unit\}$.
   By Fact \ref{fact:nonEmptyD} we have $\cDnt(\JId{L}) \iso \cD(\JI{L}) \iso L$, and
   by Theorem \ref{thm:mainIsoDownsets} and Lemma \ref{lem:quotient} that 
   $L \iso \cDnt(D)) \iso \cDnt(P_D)$ where $D = \gD(L)$. So by 
   Corollary \ref{cor:revBirkEmbEx} we get the following.  
   \begin{corollary}
     For a sublattice $L$ of a product $\PC = \prod \CD$,
     $P_D \iso \JId{L}$. 
   \end{corollary}
   
    
     \subsection{Embeddings of $L$  and homomophisms to $\JId{L}$}

       Let $e:L \to \PC$ be an embedding of a lattice into a product of 
       chains, so Corollary \ref{cor:1to1} yields a corresponding
       spanned extension $D = \gD(e(L))$ of $\dC$, which we again
       view as a pointed chain decompositon $\phi_0: \dC \to D$.  
       Composing with $q: D \to P_D  \iso  \JId{L} =: J^\infty$, we get
       a surjective homomorphism $\phi_e = q \circ \phi_0: \dC \to J^\infty$.

       Using the isomorphsim $L \iso \cDnt(J^\infty)$ we get an explicit
       formula for $\phi_e$, and show that $e \mapsto \phi_e$ is a one-to-one
       correspondence.  For this section  we denote a principle downset
       $\id x$ of $J^\infty$ by $\pid{x}$, and its unique predecessor
       $\id x \setminus \{x\}$ in $\cDnt(J^\infty)$ by $\ppid{x}$.
       Similarily we use $\pid{\cdot}_\PC$ and $\ppid{\cdot}_\PC$
       for downsets in $\PC$.

       For any surjective homomorphism $\phi: \dC \to J^\infty$, let
       \begin{equation}\label{eqn:ephi}       
    e_\phi: \cDnt(J^\infty)\to\PC: \pid{x} \mapsto \bigvee \phi^{-1}(\pid{x}).
       \end{equation}
       For an embedding $e:\cDnt(J^\infty)\to\PC$, let 
      $\phi_e: \dC \to J^\infty$ be defined by
      \begin{equation}\label{eqn:phie}
  \phi_e^{-1}(x) = [\pid{ e \pid{x}}_\PC \setminus
       \ppid{ e \pid{x}}_\PC ]  \cap \dC.
      \end{equation} 
     
     We are ready for our generalisation of Corollary \ref{cor:La}.   

      \begin{theorem}\label{thm:one2one5}
       For any distributive lattice $L$, let $J^\infty = \JId{L}$.
       The maps $\phi \to e_\phi$ and $e \to \phi_e$ defined in 
       equations \eqref{eqn:ephi} and \eqref{eqn:phie} give a one-to-one
        correspondence between embeddings of
       $\cDnt(J^\infty) \iso L$ into products of chains,
        and surjective homomorphisms
       of pointed union of chains to $J^\infty$. 
     \end{theorem}
     \begin{proof} 

     For a downset $\pid{x}$ we have
     $\phi^{-1}_e\pid{x} = \pid{ e \pid{x}}_\PC \cap \dC$, and so
     $e =e_{\phi_e}$ is immediate. Indeed, recalling for the last equality that
     $p \to \id p \cap \dC$ and $S \to \bigvee S$  are inverse :
         \[ e _{\phi_e} \pid{x} = \bigvee \phi^{-1}_e \pid{x}  = \bigvee \pid{ e \pid{x}}_\PC \cap \dC =e \pid{x}. \]

       To see that $\phi_{e_\phi} = \phi$, let $x$ be in $J^\infty$. Using in the third line that $\pid{p} \setminus \ppid{p} = p$ for any $p$, we get 
        \begin{eqnarray*}
           \phi^{-1}_{e_\phi}(x)
         & = & \left[ \pid{e_\phi\pid{x}}_\PC \setminus \ppid{e_\phi\pid{x}}_\PC \right] \cap \dC \\
         & = & \left[ \pid{\bigvee \phi^{-1}\pid{x}}_\PC\setminus \ppid{\bigvee \phi^{-1}\pid{x}}_\PC  \right] \cap \dC \\
         & = & \left[\bigvee \phi^{-1}\pid{x}\right] \cap \dC \\
         & = & \phi^{-1}(x).
       \end{eqnarray*} 
        
     \qed\end{proof}






   \subsection{Classification of Embeddings}

    For an embedding $e: L \to \PC$, we have a surjective homomorphism 
       $\phi_e = q \circ \phi_0: \dC \to \gD(e(L)) \to J^\infty$. 
    We observed for Corollary \ref{cor:MainTermRevFull} that when $e$ is a 
    $\full$-embedding then be removing extremal elements,  
    $\phi_0$ induces a homorphism of $\CD^*$ to $D^*$. This extends by $q$
    to a surjection $\phi_e: \CD^* \to \JI{e(L)}$. 
    Again by Lemma \ref{lem:full}, $e$ is subdirect if and only if
    $D = \gD(e(L)) $
    has no down edges, which is true if and only if all $\phi_e$ is
    injective on all chains of $\CD^*$.
    It is tight if and only if $D$ contains no cycles, which is true
    if and only if $\phi_e$ is injective.  

 We have thus showed the following corollary to Theorem~\ref{thm:one2one5}.

      \begin{corollary}\label{cor:one2one5}
         The constructions $e \to \phi_e$ and $\phi \to e_\phi$ give 
         a one-to-one correspondence between full embeddings of $L$ into
         products of chains, and surjective homomorphism of disjoint unions
         of chains to $\JI{L}$. Subdirect embeddings correspond with 
         homomorphisms that are injective on each chain, and tight embeddings
         correspond with bijections-- so-called chain decompositions. 
      \end{corollary}


       


  \section{Sublattices of $\PC$ as Downset or Independent set Lattices}\label{sec:IndSet}

  In this section, when we have a product $\PC = \prod C_i$
  of a family $\CD = \{C_1, \dots, C_d\}$ of chains, $\cT$
  will be the {\em union of tournaments} we get from the
  parallel sum of the chains in $\CD$ by removing the loops
  from all vertices. 
  As we will be now referring to elements of a chain $C_i$ individually,
  we afix a superscript, for example,  denoting the element $3$ or $x_j$,
  of $C_i$ by $3^{(i)}$ or $\x{i}_j$ respectively.

     \begin{definition}[$\gA(\irrFam)$ and $\gA(L)$]\label{const:A}
       For a family $\CD = \{C_1, \dots, C_d\}$ of chains
       and a family of irreducible intervals $\irrFam$ of 
      $\PC = \prod \CD$,  let $\gA(\irrFam)$ be the digraph we
      get from $\cT$ by adding the arc
        $(\ajb, \aia)$ for each $\IIs \in \irrFam$.
      For a sublattice $L$ of $\PC$, let $\gA(L) = \gA(\irrFam_L)$. 
      (See Figure \ref{fig:RivEx3} for examples of $\gA(L)$ corresponding
       to the sublattices $L$ of $\PC$ in Figure \ref{fig:RivEx2}. The 
       digraph $\gA(L)$ is still the transitive closure of the depicted
       digraph, where all unoriented arcs are oriented up, but it is no
       longer reflexive. )
    \end{definition}

 \begin{figure}
   \centering{
    \setlength{\unitlength}{.9mm}%
   \begin{picture}(120,50)(0,0)%
    \put(0,8){\includegraphics[scale=.63]{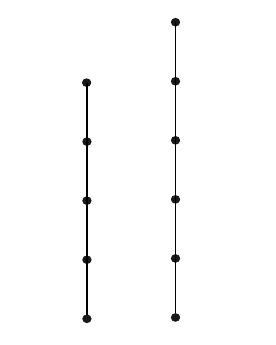}}%
    \put(7,3){$\cT = \gA(\PC)$} 
    \put(4,38){$\ai41$}
    \put(22,45){$\ai52$}
    \put(4,10){$\ai01$}
    \put(22,10){$\ai02$}
    \put(30,8){\includegraphics[scale=.63]{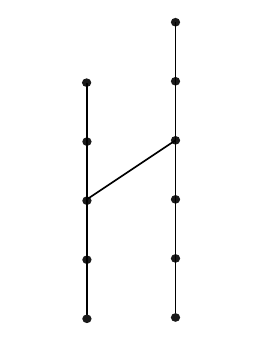}}%
    \put(40,3){$\gA(L_1)$}
    \put(34,24){$\ai21$} 
    \put(52,31){$\ai32$}
    \put(34,38){$\ai41$}
    \put(52,45){$\ai52$}
    \put(34,10){$\ai01$}
    \put(52,10){$\ai02$} 
    \put(5,0){
    \put(55,8){\includegraphics[scale=.63]{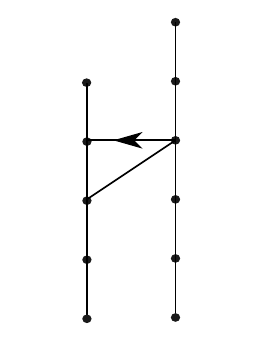}}%
    \put(65,3){$\gA(L_2)$} 
    \put(77,31){$\ai32$}
    \put(59,38){$\ai41$}
    \put(77,45){$\ai52$}
    \put(59,10){$\ai01$}
    \put(77,10){$\ai02$}}
    \put(8,0){
    \put(80,8){\includegraphics[scale=.63]{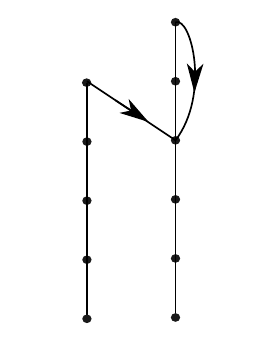}}%
    \put(90,3){$\gA(L_3)$} 
    \put(102,31){$\ai32$}
    \put(84,38){$\ai41$}
    \put(103,45){$\ai52$}
    \put(84,10){$\ai01$}
    \put(102,10){$\ai02$}}
    
  \end{picture}%
  }
   \caption{ Sublattices of $\PC = Z_4 \times Z_5$ and their digraphs $\gA$}
    \label{fig:RivEx3}  
  \end{figure}

  A directed path $x_1 \to x_2 \to \dots \to x_d$ in a 
 digraph is {\em non-trivial} if $d \geq 1$.  
 Observe that the directed path $x$ is trivial, while the
 directed path, or {\em loop}, $x \to x$, is non-trivial.

  For a vertex $x = (x_1, \dots, x_d)$ of $\PC$ we let $I_x$ be the set
  \[ I_x: =  \{\x{1}_1, \dots, \x{d}_d\}.\]

  \begin{definition}\label{def:indSet}
   A subset $I$ of a digraph $D$ is {\em independent} if 
   there is no non-trivial $xy$-path in $D$ for
   $x$ and $y$ in $I$.
  \end{definition}

  Note that this is a much stronger notion of independence than is usually defined
  for digraphs, it implies independence in the transitive closure of the graph. 
  In fact, no vertex that is in a cycle, including looped vertices, 
  can be in an independent set.
  So for a digraph like $\cT$  which we got by removing loops from a poset,
  the independent sets are exactly the antichains of the poset.
  For $\cT$ in particular, the $d$ element independent sets are exactly 
  $I_x$ for $x \in \PC$. 

  \begin{definition}[$\Ind_d(A)$]\label{def:indFam}
  Let $\Ind_d(A)$ be the family of independent sets of size $d$ 
  of a digraph $A$, and order it as follows. 
  For $I,I' \in \Ind_d(A)$, let 
  $I \leq I'$ if for each $a' \in I'$ there is an
  $aa'$-path in $A$ with $a \in I$. 
  \end{definition}

  When restricting to spanned extensions $A$ of $\cT$, for which all 
  $d$-element subsets are of the form $I_x$, this ordering
  has a simpler definition, observed in \cite{Ko83} in the case that $A$
  is a poset.

  \begin{lemma}\label{lem:ordEasy} 
     For a spanned extension $A$ of $\cT$ and $I_x, I_y \in \Ind_d(A)$,
     we have  $I_x \leq I_y$ if and only $x_i \leq y_i$ for all $i$.   
   \end{lemma}
   \begin{proof}
     The 'if' direction is immediate from Definition \ref{def:indFam} as
     all arcs of $\cT$ are in $A$. For the 'only if' direction, assume that
     $I_x \leq I_y$, and let $i \in [d]$.  We must show that $x_i \leq y_i$.
     Assume that it is not, so $y_i < x_i$, and so there is a directed path in 
     $A$ from $y_i$ to $x_i$. As $I_x \leq I_y$ there is a path
     in $A$ from $x_i$ to $y_j$ for some $j \in [d]$. If $i = j$ the $x_i$ and $y_i$
     are in a directed cycle, contradicting the fact that $x_i$ is in an independent
     set.  If $i \neq j$ then there is a path from $y_i$ to $x_i$ and one from $x_i$
     to $y_j$, contradicting the fact that $y_i$ and $y_j$ are in an independent set.   \qed\end{proof}

   The following is thus clear.

   \begin{lemma}\label{lem:baseIndSetIso}
       Where $\PC = \prod \CD$, the map  
      $I: \PC \to \Ind_d(\cT): x \mapsto I_x$
      is an isomorphism. 
   \end{lemma}

   Now, consider removing an irreducible interval to make a
   sublattice $L$ of $\PC$. Figure \ref{fig:RivEx3}
   shows which arcs we need to add to $\cT = \gA(\PC)$
   to make a digraph $\gA(L)$ so that the isomorphism 
   $L \iso \Ind_d(\gA(L))$ is preserved.
   Recall that $L_1 = \PC \setminus \II2132$.  
   To ensure that, for example,
   $I_{(2,3)} = \{\ai21, \ai32 \}$ is no longer independent, 
   we add the arc $(\ai21,\ai32)$ to $\cT$.
  This also ensures that such sets
   an $I_{(2,4)}$ and $I_{(1,5)}$  are no longer independent. 

   \begin{remark}\label{note:closedA}
   Though $L_3$ is $\PC \setminus \II2134$
   we 
   added the arcs $(\ai41,\ai32)$ and $(\ai52,\ai32)$.
   This is because the intervals $\II2134$ and $\II2235$ are
   the same. 
   \end{remark}




    The proof of the following is very similar to that of Theorem
    \ref{thm:mainIsoDownsets}. 

    \begin{theorem}\label{thm:mainIsoAntichain}
        Let $L$ be sublattice of $\PC = \prod \CD$,
        where $\CD = \{C_1, \dots, C_d\}$ is a family of disjoint chains, and
        let $A = \gA(L)$.
        The map 
           \[ I :L \to \Ind_d(A): x \mapsto I_x:= \{\ai{x_1}1, \dots, \ai{x_d}d \} \]
    is a lattice isomorphism.
    \end{theorem}

    \begin{proof}

     As adding arcs to $\cT$ or a spanned extension cannot not create new
     independent sets, we have that $\Ind_d(A)$ is a subset of $\PC =  \Ind_d(\cT)$;
     and so by Lemma \ref{lem:ordEasy} it is a subposet. 
    
     Thus it is enough to show that  $I: L \to \Ind_d(A)$ is a bijection.
     We do this by induction on the size of $A$.
     In the case that $A = \cT$, we have $L = \PC$, 
     so the isomorphism is given in Lemma \ref{lem:baseIndSetIso}.
     Now let $A = A' \cup \{(\ajb, \aia)\}$.     

       We must show that an independent set $I_x \in \Ind_d(A')$ is not in 
      $\Ind_d(A)$ if and only if $x$ is in $\IIs$. That is, we must show that
      $I_x \in \Ind_d(A') \setminus \Ind_d(A)$ if and only if $x_i \geq \alpha$ and
       $\beta \geq x_j$.  For $I_x \in \Ind_d(A')$, clearly $I_x \not\in \Ind_d(A)$ if
       and only if $(\ajb, \aia)$ completes a path between 
       $x_j$ and $x_i$, so if and only if $x_i \geq \alpha$ and $\beta \geq x_j$, 
       as needed.   \qed\end{proof}

     Taking the inverse isomorphism $I^{-1}$ in Theorem
     \ref{thm:mainIsoAntichain} we get an 
     `independent set analogue' to Corollary \ref{cor:MainTermRev}.
     A special case of it was implicit in \cite{Ko83}.  
     \begin{corollary}\label{cor:RevEmbAntichain}
       For any spanned extension $A$ of $\cupT$ 
       the map $I^{-1}: \Ind_d(A) \to \PC$
       is an embedding of $\Ind_d(A)$ into $\cP$.       
     \end{corollary}
    
    \subsection{Consequences}
   
    It is not hard to see that for a given family $\irrFam$ of irreducible
    intervals, we get from $\gD(\irrFam)$ to 
    $\gA(\irrFam)$ by 'dropping the source vertex of every
    arc by one'. 
    We do this for closed families $\irrFam$ with the following construction. 

   \begin{definition}\label{def:constK}
     For a  spanned extension $D$ of $\dC$,  
     let $A = \gK(D)$ be the digraph on the vertices of 
     $\cT$ with arcset    
     \[ \{ (\ajb, \jis)  \in \cT^2 \mid (\jibp, \jis) \in D \}. \]
   \end{definition}

   Some points to keep in mind for $A = \gK(D)$ are: 
   \begin{enumerate} 
    \item We have $\ji0i = \ji0j$ in $D$, but  $\ai0i \neq \ai0j$ in $A$.
    \item Arcs in $D$ originating at $\zero$ do not yield arcs in $A$.
    \item Arcs in $D$ originating at $\unit$ yield $d$ arcs in $A$
          (this agrees with what is observed in Remark \ref{note:closedA}).      
   \end{enumerate}

    Applying $\gK$ to the digraphs $\gD$ in Figure \ref{fig:RivEx2}
   which, recall,  are reflexive,  yields the corresponding digraphs $\gA$ in
   Figure \ref{fig:RivEx3}. The graphs in 
   Figure \ref{fig:RivEx3} are irreflexive
   except  $\gA(L_3)$ which, being transitive,
   has loops on the vertices $\ai32, \ai42,$ and $\ai53$. 
   Checking easily that $\gK(\dC) = \cT$ the following is
   an immediate consequence of Definitions \ref{const:D}
   and \ref{const:A}.
   We require $\irrFam$ to be closed because of point 3.
   mentioned just above.  
   
   \begin{fact}\label{fact:Kconst}
     If $\irrFam$ is a closed family of irreducible
     intervals then $\gA(\irrFam) = \gK(\gD(\irrFam))$.
   \end{fact} 
   
  Observe that this construction is clearly invertible, though only if defined
  on spanned extensions $A$ of $\cT$ coming from closed families $\irrFam$. 
  Its inverse  is essentially a generalisation of  Koh's construction from
   \cite{Ko83};
  the only difference is that Koh's construction would 
  have loops on $\cT$. As he considered only posets,
  this difference would be cosmetic for him, but it is
  significant for us. 

  The construction $\gK$ is useful in proving results such as the following.

    \begin{proposition}\label{prop:Dtrans}
       Let $A = \gA(\irrFam)$.  
       If $\irrFam$ is closed then $A$ is transitive.
    \end{proposition}
    \begin{proof}
      Assume that $\irrFam$ is closed. Then by Proposition \ref{prop:Rtrans}
      we have that $D = \gD(\irrFam)$ is transitive, and by Fact \ref{fact:Kconst}
      that  $A = \gK(D)$.   
      Now assume that $\ai\alpha{i} \to \ai\beta{j}$ and $\ai\beta{j} \to \ai\gamma{k}$
      in $A$. Then $\ai{(\alpha+1)}{i} \to \ai\beta{j}$ and 
      $\ai{(\beta+1)}{j} \to \ai\gamma{k}$ in $D$. 
      As we always have $\ai\beta{j}\to \ai{(\beta+1)}{j}$,
      we get an $(\ai{(\alpha+1)}{i}, \ai\gamma{k})$-path in $D$. By the transitivity
      of $D$, $ \ai{(\alpha+1)}{i} \to \ai\gamma{k}$ in $D$, and so 
      $\ai\alpha{i} \to \ai\gamma{k}$ in $A$, as needed for transitivity. 
    \qed\end{proof}

    \begin{lemma}\label{lem:Jacyclic}
      Let $L$ be a sublattice of $\PC$, and $A = \gA(L)$. Then 
      $L$ is subdirect if and only if $A$ is irreflexive, which is true
      if and only if $A$ is acyclic. 
    \end{lemma}
    \begin{proof}
      We have from Lemma \ref{lem:full} that $L$ is subdirect if and only
      if $\gD(L)$ has no edges of the form $\ji\alpha{i} \to \ji{(\alpha-1)}i$.
      As $A = \gK(\gD(L))$, this is true if and only if $A$ is irreflexive.
      As $A = \gA(\irrFam)$ for a closed family $\irrFam$, it is transitive,
      so this is true if and only if $A$ is acyclic.  
    \qed\end{proof}

     Similar to what we did in Section \ref{sect:Koh}
     we may view $\cupT$ as a {\it tournament decomposition} of $A$ rather
     than viewing $A$ as a spanned extension of $\cupT$, and get the following.
 
    \begin{corollary}
     For every digraph $A$ such that $\Ind_d(A) \iso L$,
     and every decomposition of $A$ into a family $\cT$ of $d$ tournaments,  
     there is an embedding of $L$ into $\PC$ such that 
     $A = \gA(L)$.
    \end{corollary}

    Specialising to the case that $A$ is acyclic (and so also irreflexive) then
    adding cosmetic loops, this gives the following result complementing
    Theorem \ref{thm:koh}. For a digraph $A$ let $A^r$ be the reflexive digraph we
    get by adding loops to every vertex.  

    \begin{corollary}\label{cor:antialldig}
      For every poset $P$ such that $\rA(P) \iso L$ there is a subdirect 
      embedding of $L$ into a product of chains $\PC$ such that $P = \gA^r(L)$.  
    \end{corollary} 
    
  

  \providecommand{\bysame}{\leavevmode\hboxto3em{\hrulefill}\thinspace}


\begin{thebibliography}{15}



   \bibitem{Bi}
   G. Birkhoff 
   {\em Rings of sets}
   Duke Math. Jour.  3 (3) (1937) pp 443–454.
   (doi:10.1215/S0012-7094-37-00334-X).
  



    \bibitem{Di50}
    R. Dilworth.
    {\em A decomposition theorem for partially ordered sets}
    Annals of Math. 51 (1950) pp. 161-166.

   \bibitem{Di60}
    R. Dilworth.
    {\em Some Combinatorial problems on partially ordered sets.}
    Combinatorial Analysis (Proc. Symp. Appl. Math.), American Math. Soc. 
    (1960), pp. 85-90.

  
   \bibitem{Gr99}
   G. Gr\"atzer
   {\em Lattice Theory: Foundation.}
   Springer Basel AG (2011).


   \bibitem{Ko83}
    K. Koh,
    {\em On the lattice of maximum-sized antichains of a finite poset.}
    Alg. Universalis, 17 (1983) pp. 73-86. 
  

    \bibitem{La98}
     R. Larson, 
     {\em Embeddings of finite distributive lattices into products of chains}
     Semigroup Forum 56 (1998) pp. 70-77.

 

    \bibitem{Ri74}
    I. Rival,
    {\em Maximal sublattices of finite distributive lattices. II.}
    Proc. Amer. Math. Soc.  44 (1974) pp. 263-268.


    \bibitem{Si15}
    M. Siggers,
    {\em Distributive Lattice Polymorphisms on Reflexive Graphs}
    Submitted (Nov. 2014) {\tt arXiv:1411.7879 [math.CO]}. 


 \end{thebibliography}
 \end{document}